\documentclass[12pt]{amsart}
\usepackage{geometry}                
\usepackage{graphicx}
\usepackage{amssymb}
\usepackage{epstopdf}
\usepackage{amsmath}
\usepackage{color}
\usepackage{hyperref}
\usepackage{pdfsync}
 \usepackage{tikz}
 \usepackage{tikz-cd}

\usepackage[all]{xy}
\xyoption{matrix}
\xyoption{arrow}

\usetikzlibrary{arrows,decorations.pathmorphing,backgrounds,positioning,fit,petri}

 \tikzset{help lines/.style={step=#1cm,very thin, color=gray},
help lines/.default=.5} 
\tikzset{thick grid/.style={step=#1cm,thick, color=gray},
thick grid/.default=1} 

\newtheorem{thm}{Theorem}[section]
\newtheorem{lem}[thm]{Lemma}
\newtheorem{cor}[thm]{Corollary}
\newtheorem{prop}[thm]{Proposition}

 \newtheorem{innercustomthm}{Theorem}
\newenvironment{customthm}[1]
  {\renewcommand\theinnercustomthm{#1}\innercustomthm}{\endinnercustomthm}

\theoremstyle{definition}
\newtheorem{defn}[thm]{Definition}

\theoremstyle{remark}
\newtheorem{rem}[thm]{Remark}
 
\numberwithin{equation}{section}

 \newcommand{\onto}{\twoheadrightarrow}

\DeclareMathOperator{\colim}{colim}%

\newcommand{\field}[1]{\mathbb{#1}}
\newcommand{\ZZ}{\ensuremath{{\field{Z}}}}
\newcommand{\CC}{\ensuremath{{\field{C}}}}
\newcommand{\RR}{\ensuremath{{\field{R}}}}

\newcommand{\commentout}[1]{}

\newcommand{\cC}{\ensuremath{{\mathcal{C}}}}
\newcommand{\cD}{\ensuremath{{\mathcal{D}}}}
\newcommand{\cE}{\ensuremath{{\mathcal{E}}}}
\newcommand{\cF}{\ensuremath{{\mathcal{F}}}}

\newcommand{\cP}{\ensuremath{{\mathcal{P}}}}

\title{Second obstruction to pseudoisotopy I}
\author{Kiyoshi Igusa}
\address{Department of Mathematics, Brandeis University, Waltham, MA 02454}\email{igusa@brandeis.edu}
\thanks{Supported by the Simons Foundation}



\begin{document}

\begin{abstract} A pseudoisotopy of $M$ is a diffeomorphism of $M\times I$ which is the identity on $M\times 0$. We give an explicit construction of pseudoisotopies of 4-manifolds which realize certain elements of the ``second obstruction to pseudoisotopy''. The next paper will give precise statements of our results for 5-manifolds which had been announced long ago.
\end{abstract}

\maketitle

\tableofcontents

\section*{Introduction}

Recall that, for any smooth compact manifold $M$, a \emph{pseudoisotopy} of $M$ is a diffeomorphism of $M\times I$ which is the identity on $M\times 0\cup \partial M\times I$. The restriction of this diffeomorphism to $M\times 1$ gives a diffeomorphism of $M$ which is the identity on $\partial M$. This diffeomorphism is said to be \emph{pseudoisotopic} to the identity diffeomorphism of $M$ rel $\partial M$. Let $\cC(M)$ be the group of pseudoisotopies of $M$. Thus
\[
\cC(M)=Dif\!f(M\times I\,rel\, M\times 0\cup \partial M\times I).
\]
We give $\cC(M)$ a suitable topology so that, up to canonical homotopy, a path in $\cC(M)$ is given by a diffeomorphism of $M\times I\times I$ which preserves the last coordinate. (Section 1.) Since $\cC(M)$ is a topological group, $\pi_0\cC(M)$ will be a group.

For $M$ simply connected with $\dim M\ge5$, Jean Cerf \cite{Cerf} showed that $\cC(M)$ is connected. For $\dim M\ge 6$, Hatcher and Wagoner \cite{HW} showed that there is an exact sequence
\begin{equation}\label{HW exact sequence}
	Wh_1^+(\pi_1M;\ZZ_2\oplus \pi_2M)\to \pi_0\cC(M)\to Wh_2(\pi)\to 0.
\end{equation}

When $\dim M\le 5$, we apply a ``suspension functor'' $\sigma:\cC(M)\to \cC(M\times I)$ several times to get to the \emph{stable pseudoisotopy space} $\cP(M)=\colim \cC(M\times I^N)$ \cite{Stability} which should be used instead of $\cC(M)$ in the above exact sequence.

In \cite{What happens}, \cite{A-infty} and \cite{I:thesis} the sequence \eqref{HW exact sequence} was extended to an exact sequence
\begin{equation}\label{extended HW exact sequence}
	K_3\ZZ[\pi_1M]\xrightarrow\chi Wh_1^+(\pi_1M;\ZZ_2\oplus \pi_2M)\xrightarrow\theta \pi_0\cP(M)\to Wh_2(\pi)\to 0.
\end{equation}

 The question arises: Which elements of $\pi_0\cP(M)$ are realized when $\dim M<6$, i.e., which elements of $\pi_0\cP(M)$ lie in the image of the stabilization map $\pi_0\cC(M)\to\pi_0\cP(M)$? This question was answered in dimension 5 by the author in a preliminary version of \cite{Wh1b} and partially answered in the topological category in dimension 4 by Kwasik in \cite{Kwasik} using the techniques of \cite{HW} and one idea from \cite{Wh1b}. In both \cite{Wh1b} and \cite{Kwasik} we need an assumption which trivializes the image of $\chi$ in \eqref{extended HW exact sequence}. 
  
 This series of papers discusses this question for $\dim M=4,5$. In the current paper we show that, for certain $4$-dimensional manifolds $M$, some elements of the second obstruction group 
 \[
 	Wh_1^+(\pi_1M;\ZZ_2\oplus \pi_2M)=Wh_1^+(\pi_1M;\ZZ_2)\oplus Wh_1^+(\pi_1M; \pi_2M)
 \]
can be realized by pseudoisotopies of $M$. Given a group $G$ and $G$-module $A$, $Wh_1^+(G;A)$ is a quotient of $A[G]/A[e]=A\otimes \ZZ[G]/\ZZ[e]$ (Definition \ref{def: Wh1+ Z}). In particular, $Wh_1^+(\pi_1M; \pi_2M)$ is generated by elements $\alpha[\sigma]$ where $\alpha\in \pi_2M$ and $\sigma$ is a nontrivial element of $\pi_1M$.

\begin{customthm}{A}[Theorem \ref{prop: realizing Wh1pi2}]\label{thm A}
Let $M$ be the connected sum of $\CC P^2$ with a nonsimply connected $4$-manifold $X$. Then the elements of $Wh_1^+(\pi_1M; \pi_2M)$ of the form $\alpha[\sigma]$ where $\alpha$ is represented by the $2$-sphere $\CC P^1\subset \CC P^2$ lift to $\pi_0\cC(M)$ and map to nontrivial elements of the kernel of $\pi_0\cP(M)\onto Wh_2(\pi)$.
\end{customthm}

The proof relies on the fact that $\CC P^1\subset \CC P^2$ is an embedded $2$-sphere which gives an element of $\pi_2M$ which is not hit by the Postnikov invariant $k_1\in H^3(\pi_1M;\pi_2M)$. We know by the formula for $\chi$ given in \cite{What happens} that this element of $\pi_2M$, together with any nontrivial element of $\pi_1M$, gives an element of the second obstruction group $Wh_1^+(\pi_1M;\pi_2M)$ whose image in $\pi_0\cP(M)$ is nontrivial. The assertion in this paper is that this second obstruction element of $\pi_0\cP(M)$ lifts to $\pi_0\cC(M)$ where $\dim M=4$.

The proof relies on the space $\cD(M)$ of ``lens-space models'' for pseudoisotopies of $M$ (Definition \ref{def: eye shaped graphic}). This is the space considered by Hatcher and Wagoner to determine the second obstruction to pseudoisotopy. We also consider the space $\cD_0(M)$ of ``marked lens-space models'' for $M$ since $\pi_0\cD_0(M)$ forms a group (Lemma \ref{lem: D0(M) is a group}).

We recall that $Wh_1^+(\pi_1M;\ZZ_2)$, is $\ZZ_2[\pi_1M]/\ZZ_2[e]$ modulo the conjugation action of $\pi_1M$. For example, when $\pi_1M$ is finite,
\[
Wh_1^+(\pi_1M;\ZZ_2)=\bigoplus_{c-1}\ZZ_2
\]
is the direct sum of $c-1$ copies of $\ZZ_2$ where $c$ is the number of conjugacy classes of elements of $\pi_1M$. In dimension 4, we do not know whether all elements of $Wh_1^+(\pi_1M;\ZZ_2)$ can be realized by pseudoisotopies of $M$ and any realization might have infinite order. We show in this paper that there are homomorphisms $\widetilde \theta_0:\pi_0 \cD_0(M)\to \pi_0\cC(M)$ and $\overline\lambda_0:\pi_0\cD_0(M)\to Wh_1^+(\pi_1M;\ZZ_2\oplus\pi_2M)$ making the following diagram commute. Furthermore, when $M=X\#\CC P^2$, $\overline\lambda_0$ maps onto $Wh_1^+(\pi_1M;\ZZ_2)$. 
\[
\xymatrix{
\pi_0\cD_0(M) \ar[d]_{\overline\lambda_0}\ar[r]^(.6){\widetilde \theta_0} &
	\pi_0\cC(M)\ar[d]\\
Wh_1^+(\pi_1M;\ZZ_2\oplus \pi_2M) \ar[r]^(.6)\theta& 	\pi_0\cP(M)
	}
\]
This proves the second main theorem:
\begin{customthm}{B}[Theorem \ref{thm: realizing Wh1+Z2}]\label{thm B}
Let $M$ be the connected sum of $\CC P^2$ with a nonsimply connected $4$-manifold $X$. Then every elements in the image of $Wh_1^+(\pi_1M;\ZZ_2)\to \pi_0\cP(M)$ lift to $\pi_0\cC(M)$.
\end{customthm}

The same proof works in a slightly more general case, when $M$ has an embedded $2$-sphere with odd self-intersection number. (Theorem \ref{thm: realizing Wh1+Z2 b}.) A topological version of this result for smooth $4$-manifolds was obtained by Kwasik \cite[Proposition 4.1]{Kwasik}.

\section{Definition of second obstruction to pseudoisotopy}

We will review the basic theory, due to Cerf \cite{Cerf}, relating diffeomorphisms of $M\times I$ and Morse theory. Then we focus on the second obstruction, due to Hatcher and Wagoner \cite{HW}. We assume in this section that $M$ is a compact connected smooth manifold of dimension $n\ge4$.

\subsection{Basic setup}

Let $\cF(M)$ be the space of all smooth functions $f:M\times I\to I$ so that $f$ is equal to projection to $I$ in some neighborhood of 
\[
\partial (M\times I)=\partial M\times I\cup M\times 0\cup M\times 1.
\]
Then $\cF(M)$ is clearly contractible: $\cF(M)\simeq\ast$. Let $\cE(M)$ be the space of all $f\in\cF(M)$ with no singularities. Then, Cerf observed that
\[
	\cE(M)\simeq \cC(M)
\]
where the homotopy equivalence $\cC(M)\to \cE(M)$ is given by sending $g:M\times I\to M\times I$ to $\pi_I\circ g:M\times I\to I$ where $\pi_I:M\times I\to I$ is the projection map. Given $f\in \cE(M)$, the corresponding element of $\cC(M)$ is constructed uniquely up to contractible choice by integrating the gradient of $f$. As a consequence
\[
	\pi_0\cC(M)\cong \pi_1(\cF(M),\cE(M)).
\]
For clarity we repeat: This bijection sends the isotopy class of $g\in \cC(M)$ to the homotopy class of the 1-parameter family $f_t$ if and only if $f_1=\pi_I\circ g$ (and $f_0=\pi_I$).

\begin{prop}\label{prop: lateral composition of ft}
Let $(f_t^a),(f_t^b)$ be elements of $\pi_1(\cF(M),\cE(M))$ corresponding to $g^a,g^b\in\pi_0\cC(M)$. Thus $f_1^a=\pi_I\circ g^a$, $f_1^b=\pi_I\circ g^b$. Then the product in $\pi_0\cC(M)$, given by composition: $g^a\circ g^b$ is $f_t^{ab}$ given by putting $f_t^b$ to the left of $f_t^a$:
\[
	f_t^{ab}=\begin{cases} f_{2t}^b & \text{if } t\le \frac12\\
    f_{2t-1}^a\circ g^b& \text{if } t>\frac12.
    \end{cases}
\]
\end{prop}

\begin{proof}
First note that the family of function $f_t^{ab}$ lies in $\pi_1(\cF(M),\cE(M))$ since, at $t=0$, $f_0^b=\pi_I$ and, at $t=\frac12$, $f_1^b=\pi_I\circ g^b=f_0^a\circ g^b$. It is also easy to see that $f^{ab}$ corresponds to $g^ag^b$ since $f_1^{ab}=f_1^a\circ g^b=\pi_I\circ g^a\circ g^b$.
\end{proof}

\begin{cor}\label{cor: inverse of ft}
The inverse of $(f_t)\in \pi_1(\cF(M),\cE(M))$ is the 1-parameter family
\[
	f_t'=f_{1-t}\circ g^{-1}:M\times I\to I
\]where $g\in \cC(M)$ is a diffeomorphism of $M\times I$ whose composition with projection $\pi_I:M\times I\to I$ is $f_1\in\cE(M)$ (so that $f_0'\circ g=f_1=\pi_I\circ g$, making $f_0'=\pi_I$.)
\end{cor}

\begin{proof} Let $f^a_t=f_t'$ and $f^b_t=f_t$. By Proposition \ref{prop: lateral composition of ft}, the product $f^{ab}=f^af^b$ at $t=1$ is
\[
	f_1^{ab}=f_1^a\circ g^b=f^b_0\circ (g^b)^{-1}\circ g^b=f_0^b=\pi_I.
\]
This corresponds to the identity in $\pi_0\cC(M)$. So, $f^a=(f_t')$ is the inverse of $f^b=(f_t)$.
\end{proof}

Here is an alternate version of composition in $\pi_0\cC(M)$ in terms of functions.

\begin{prop}\label{prop: stacking functions}
The group structure (composition) on $\pi_0\cC(M)$ corresponds to stacking of functions $f_{a},f_{b}:M\times I\to I$ in $\cE(M)$ given by ``putting $f_{b}$ on top of $f_{a}$'':
\[
	(f_{a}\bullet f_{b})(x,t)=\begin{cases} \frac12 f_{a}(x,2t) & \text{if } t\le \frac12\\
   \frac12 +\frac12 f_{b}(x,2t-1) & \text{if } t>\frac12
    \end{cases}
\]
\end{prop}

\begin{proof}
Let $g_a,g_b\in \cC(M)$ be given by $g_a(x,t)=(h_a(x,t),f_a(x,t))$ and $g_b=(h_b,f_b)$. Then $g_{a},g_{b}$ are isotopic to $g_{a}',g_{b}'\in \cC(M)$ given by ``pushing $g_{a}$ down and $g_{b}$ up'':
\[
	g_{a}'(x,t)= \begin{cases} (h_{a}(x,2t),\frac12f_{a}(x,2t)) & \text{if } t\le \frac12\\
   (r(x),t) & \text{if } t> \frac12
    \end{cases}
\]
where $r\in Dif\!f(M)$ is the restriction of $g_{a}$ to $M\times 1$, i.e., $g_{a}(x,1)=(r(x),1)$ and
\[
	g_{b}'(x,t)= \begin{cases} (x,t) & \text{if } t\le \frac12\\
   (h_{b}(x,2t-1),\frac12+\frac12f_{b}(x,2t-1)) & \text{if } t> \frac12
    \end{cases}
\]
Then $g_{a}\circ g_{b}$ is isotopic to $g_{a}'\circ g_{b}'$ given by
\[
	(g_{a}'\circ g_{b}')(x,t)= \begin{cases} (h_{a}(x,2t),\frac12f_{a}(x,2t)) & \text{if } t\le \frac12\\
   (rh_{b}(x,2t-1),\frac12+\frac12f_{b}(x,2t-1)) & \text{if } t> \frac12
    \end{cases}
\]
Taking the second coordinate gives $f_{a}\bullet f_{b}$.
\end{proof}

The group structure on $\pi_1(\cF(M),\cE(M))$ is also given by stacking functions.

\subsection{The graphic}

The \emph{fiberwise singular set} $\Sigma(f)$ of $f$ is the set
\[
	\Sigma(f):=\{(x,s,t)\in M\times I^2\,:\, (x,s)\in M\times I \text{ is a singularlity of } f_t\}.
\]
 Since each function $f_t$ will generically have a finite singular set, and $f_0,f_1$ are nonsingular, $\Sigma(f)$ will, generically, be a closed 1-dimensional manifold. If $\Sigma(f)$ is empty or if $f$ is homotopic relative to $\partial (M\times I^2)$ to a family of function with empty singular set, then $f$ represents the trivial element of $\pi_1(\cF(M),\cE(M))$. Thus, $\pi_1(\cF(M),\cE(M))$ is the obstruction to the elimination of the singular set $\Sigma(f)$ by deformation of $f$.

We visualize $\Sigma(f)$ via its \emph{graphic} which is defined to be the set of all $(s,t)\in I^2$ so that $s$ is a critical value of $f_t$. Two examples are given in Figure \ref{Figure graphic}. The graphic on the right side of Figure \ref{Figure graphic} is called an ``eye'' or ``lens''. 

\begin{defn}\label{def: eye shaped graphic}
A \emph{lens-shaped graphic} or \emph{eye} is, by definition, the graphic of a family of functions $f_t:M\times I\to I$ so that $f_t$ is Morse with two critical points $p_t,q_t$ with $f_t(p_t)<f_t(q_t)$ for $t_0<t<t_1$, $f_{t_0},f_{t_1}$ have only one critical point called a \emph{birth-death point} and $f_t$ is nonsingular for $t\notin[t_0,t_1]$. A family of functions $(f_t)$ having a lens-shaped graphic will be called a \emph{lens-shaped model} for $M$.
\end{defn}
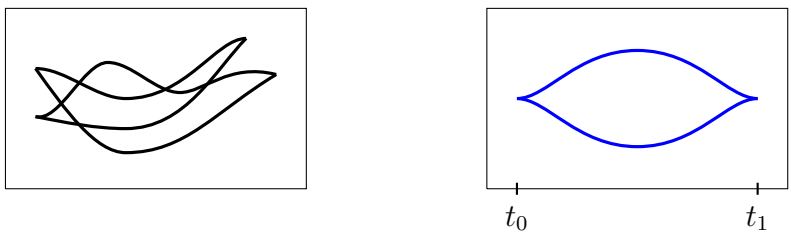
\begin{figure}[htbp]
\begin{center}
\begin{tikzpicture}[scale=.8]
\begin{scope}[xshift=-4cm]
\draw (0,0.5)--(5,0.5)--(5,3.5)--(0,3.5)--(0,0.5);
\begin{scope}[yshift=-4mm]
\draw[very thick] (.5,2.9)..controls (1,2.2) and (1.5,1.5)..(2,1.5)..controls (3,1.5) and (3.5,2.2)..(4.5,2.8);
\draw[very thick] (.5,2.1)..controls (.9,2) and (1.3,3)..(1.7,3)..controls (2.1,3) and (2.5,2.5)..(2.9,2.5)..controls (3.3,2.5) and (3.7,3)..(4.5,2.8);
\end{scope}
\draw[very thick] (.5,2.5)..controls (1,2.5) and (1.5,2)..(2,2)..controls (3,2) and (3.5,3)..(4,3);
\draw[very thick] (.5,1.7)..controls (1,1.6) and (1.5,1.5)..(2,1.5)..controls (3,1.5) and (3.5,2.5)..(4,3);
\end{scope}
\begin{scope}[xshift=4cm]
\draw (0,0.5)--(5,0.5)--(5,3.5)--(0,3.5)--(0,0.5);
\draw[very thick,color=blue] (.5,2)..controls (1,2) and (1.5,2.8)..(2.5,2.8)..controls (3.5,2.8) and (4,2)..(4.5,2);
\draw[very thick,color=blue] (.5,2)..controls (1,2) and (1.5,1.2)..(2.5,1.2)..controls (3.5,1.2) and (4,2)..(4.5,2);
\draw[thick] (.5,.6)--(.5,.4)node[below]{$t_0$} ;
\draw[thick] (4.5,.6)--(4.5,.4)node[below]{$t_1$} ;
\end{scope}
\end{tikzpicture}
\caption{Two possible graphics of 1-parameter families of functions on $M\times I$. The one on the right is called an ``eye'' or ``lens''.}
\label{Figure graphic}
\end{center}
\end{figure}

\subsection{Second obstruction}

There are two steps in the process of simplification of the singular set $\Sigma(f)$ of $f$. Hatcher and Wagoner \cite{HW} showed that, for $\dim M\ge 6$, there is a surjective homomorphism
\[
	p: \pi_0\cC(M)\to Wh_2(\pi_1M)
\]
where, for any group $\pi$, $Wh_2(\pi)$ is a quotient of $K_2(\ZZ[\pi])$.
They called this invariant the ``first obstruction to pseudo-isotopy''. 

Hatcher and Wagoner also showed that the kernel of the map $p$ is generated by ``lens-shaped models'' (Definition \ref{def: eye shaped graphic}) in the middle indices $k,k+1$ where $k=\lfloor \frac n2\rfloor$ is the greatest integer $\le \frac n2$ where $n=\dim M$. The ``second obstruction to pseudo-isotopy'' is an invariant defined on this set of 1-parameter families of function. 

To define these invariants, we will add extra structure, called ``markings'', to these families of functions, define invariants on the marked lens-shaped models, then determine how these invariants change when we change the markings.

\begin{defn}\label{def: lens-shaped models}
Let $\cD(M)$ denote the space of lens-shaped models $(f_t)$ for $M$ with critical points $p_t,q_t$ in indices $k,k+1$ where $k=\lfloor \frac n2\rfloor$. Let $\cD_0(M)$ denote the space of \emph{``marked'' lens-shaped models} for $M$. These are defined to be lens-shaped models $(f_t)$ together with:
\begin{enumerate}
\item A homotopy class of paths $\gamma$ from a fixed base point $\ast=(x_0,0,0)\in M\times I^2$ to the birth point $b_0$ of $f_t$ (at $t=t_0$ in Fig. \ref{Figure graphic}).
\item An orientation of the negative eigenspace of $D^2f_t$ at the birth point $b_0$.
\end{enumerate} 
The group $\pi_1(M)\times\ZZ_2$ acts on the set of marking, $\pi_1(M)$ acting on the paths by composition and $\ZZ_2$ acting by reversing the orientation. We denote by $\widetilde\theta: \pi_0\cD(M)\to \pi_0\cC(M)$ the set map which sends a lens-shaped model $(f_t)$ to the corresponding element of $\pi_0\cC(M)\cong \pi_1(\cF(M),\cE(M)) $. Composing with the surjective mapping $\pi_0\cD_0(M)\onto \pi_0\cD(M)$ we get a mapping
\[
	\widetilde\theta_0: \pi_0\cD_0(M)\to \pi_0\cC(M)
\]
which we will show is a homomorphism of groups (Lemma \ref{lem: D0(M) is a group}).
\end{defn}

The group structure on $\pi_0\cD_0(M)$ is given as follows where, for convenience of notation, we double the size of the parameter interval to $[0,2]$. Let $f^a,f^b$ be two marked lens-shaped models for $M$. These are 1-parameter families of functions $f_t^a,f_t^b:M\times I\to I$. The new function $f^a\ast f^b$ will be given by rescaling a family of function $f_t^{ab}:M\times I\to I$, $t\in[0,2]$ by $(f^a\ast f^b)_t=f^{ab}_{2t}$. The functions $f_t^{ab}$ are constructed in two stages.
\begin{enumerate}
\item Starting with $f^{c}=f^a\cdot f^b$ given by putting $f^b$ to the left of $f^a$ as in Proposition \ref{prop: lateral composition of ft} but with the parameter domain stretched out. Thus $f^{c}_t=f^b_t$ for $t\le 1$ and $f^{c}_t=f_{t-1}^a\circ g^b$ for $t>1$ where $g^b\in \cC(M)$ so that $f^b_1=\pi_I\circ g^b$. See Fig. \ref{Fig: composition in D0(M)}.
\item ``Merging'' the birth point $b_2$ of $f^a\circ g^b$ (related to the birth point $b_0'$ of $f^a$ by $b_2=(g^b)^{-1}(b_0')$) with the death point $b_1$ of $f^b$ along the dashed path $w_t$ in Fig. \ref{Fig: composition in D0(M)} using the markings. The result is $f_t^{ab}$ for $t\in [0,2]$.
\end{enumerate}


%
\begin{figure}[htbp]
\begin{center}
\begin{tikzpicture}[scale=.75]
\begin{scope}[xshift=-5cm]
\draw (6.5,0.5)--(6.5,3.5)--(-1,3.5)--(-1,0.5);
\draw (6.5,3.5)--(14,3.5)--(14,0.5)--(6.5,0.5);
\draw[very thick,color=blue] (.5,2)..controls (1,2) and (1.5,2.8)..(2.5,2.8)..controls (3.5,2.8) and (4,2)..(4.5,2);
\draw[very thick,color=blue] (.5,2)..controls (1,2) and (1.5,1.2)..(2.5,1.2)..controls (3.5,1.2) and (4,2)..(4.5,2)
node[below]{$b_1$};
\draw[thick] (-1,.6)--(-1,.4)node[below]{$0$} ;
\draw (2.5,2) node{$f^{b}$};
\draw[->,thick] (-1,0.5)--(2,0.5);
\draw (2,.5) node[below]{$t$};
\draw[->,thick] (-1,0.5)--(-1,2.5);
\draw (0,2.5) node[left]{$s$};
\draw[thick] (4.5,.6)--(4.5,.4)node[below]{$t_1$} ;
\draw[blue] (.5,2) node[below]{$b_0$};
\draw[blue] (3.5,1.5) node[below]{$p_t$};
\end{scope}
\draw (-5,0.5)--(8.5,0.5);
\draw[thick,blue,dashed] (3.5,2)--(-.5,2);
\draw[blue] (1,2) node[above]{$w_t$};
\begin{scope}[xshift=3cm]
\draw (2.5,2) node{$f^{a}\circ g^b$};
\draw[very thick,color=blue] (.5,2)..controls (1,2) and (1.5,2.8)..(2.5,2.8)..controls (3.5,2.8) and (4,2)..(4.5,2);
\draw[very thick,color=blue] (.5,2)..controls (1,2) and (1.5,1.2)..(2.5,1.2)..controls (3.5,1.2) and (4,2)..(4.5,2)
(.5,2) node[below]{$b_2$};
\draw[thick] (-1.5,.6)--(-1.5,.4)node[below]{$1$} ;
\draw[thick] (.5,.6)--(.5,.4)node[below]{$t_2$} ;
\draw[thick] (6,.6)--(6,.4)node[below]{$2$} ;
\end{scope}
\end{tikzpicture}
\caption{The family of functions $f^c_t$ on $M\times I$ for $t\in [0,2]$ is given to be $f^b_t$ on the left side where $t\in[0,1]$ and $f_{t-1}^a\circ g^b$ on the right side where $t\in [1,2]$. The death point $b_1$ at $t_1$ is cancelled with the birth point $b_2$ at $t_2$ along the path $w_t$ to produce the marked lens shaped model $f_t^{ab}$.}
\label{Fig: composition in D0(M)}
\end{center}
\end{figure}
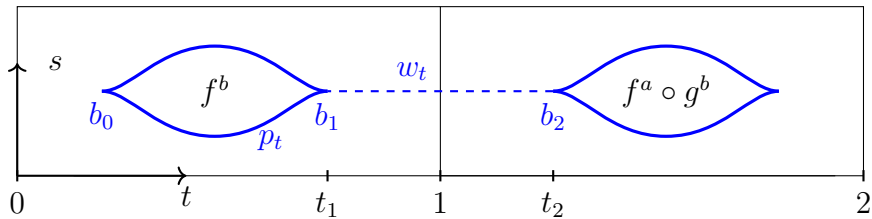

The family of functions $f_t^c$, $t\in[0,2]$, is indicated in Figure \ref{Fig: composition in D0(M)}. The ``merging'' process is given as follows. We have functions $f^c_t$ for $t\in [t_1,t_2]$ which are nonsingular for $t_1<t<t_2$ and so that each of $f^c_{t_1},f^c_{t_2}$ has exactly one critical point $b_1,b_2$, resp., which is a birth-death point of index $k=\lfloor \frac n2\rfloor$. Furthermore, the markings on $f^b,f^a$ (when we compose with $(g^b)^{-1}$ in the case of $f^a$) give paths from $b_1,b_2$ to a common basepoint. In the case of $b_1$ this path goes first along the path $p_t$ of index $k$ critical points of $f_t^b$ to the birth point $b_0$ of $f^b$ and then along the path to the basepoint given by the marking. We combine these to give a path $w_t \in M\times I$ for $t\in [t_1,t_2]$ so that $w_t=b_i$ for $t=t_i$ and $f^c_t$ is nonsingular at $w_t$ for all $t\in (t_1,t_2)$.

At each point $w_t$ for $t\in [t_1,t_2]$ we let $v_t$ be the tangent vector which points in the direction of the gradient of $f^c_t$ for $t\in (t_1,t_2)$ and which points in the positive cubic direction of $f_t^c$ in the kernel of $D^2f_t^c$ for $t=t_1,t_2$. Using the second part of the markings for $f^a,f^b$ (composed with $D(g^b)^{-1}$ in the case of $f^a$), we have an oriented tangent $k$ plane at $b_1,b_2$ which is orthogonal to $v_t$. For $n\ge3$ the space of oriented $k$-planes in $\RR^n$ is simply connected. So, this can be extended uniquely up to homtopy to a family of oriented tangent $k$ planes along $w_t$ orthogonal to $v_t$ at each $w_t$. 

We can choose local coordinates $(x,y,z)\in \RR^k\times \RR^{n-k}\times \RR$ for $M\times I$ centered at $w_t$ so that
\[
	f^c_t(x,y,z)=h_t(z)+||y||^2-||x||^2
\]
for all $t\in [t_1,t_2]$ where $h_t(z)$ is a smooth family of functions in one variable $z$ so that $h_t(0)=f^c_t(w_t)$ for each $t\in [t_1,t_2]$. Working only with $h_t$ we can modify the family of function $f_t^c$ in a small neighborhood of the path $w_t$ so that the new function has a birth-death point at $w_t$ for all $t\in[t_1,t_2]$. We can then resolve these degenerate critical points into two Morse critical points in ``cancelling position''. This gives the family of functions $f_t^{ab}$ which is well-defined up to homotopy. The product is $(f^a\ast f^b)_t=f_{2t}^{ab}$. The marking for $f^a\ast f^b$ is given by the marking for $f_t^b$.

\begin{lem}\label{lem: D0(M) is a group}
For $n\ge 3$, $\pi_0\cD_0(M)$ forms a groups as outlined above and the mapping $\widetilde\theta_0: \pi_0\cD_0(M)\to \pi_0\cC(M)$, given in Definition \ref{def: lens-shaped models} is a group homomorphism.
\end{lem}

\begin{proof}
The multiplication rule is $(f^a)(f^b)=f^a\ast f^b$ as described above. It is straightforward to show that this is associative. It is well-defined since the operation is given by the path in the space of oriented $k$-planes in the kernel of $df_t$ which is given by the marking. Since this multiplication rule is a deformation of the product given in Proposition \ref{prop: lateral composition of ft}, the mapping $\widetilde\theta_0: \pi_0\cD_0(M)\to \pi_0\cC(M)$ is a homomorphism of semigroups.

The identity in $\pi_0\cD_0(M)$ is given by any lens-shaped model $(f_t)$ which is symmetrical in the sense that $f_t=f_{1-t}$ for all $t$. The left inverse of $(f_t^b)$ is $f_t^a$ given by
\[
	f_t^a=f_{1-t}^b\circ g^{-1}:M\times I\to I
\]
where $g\in \cC(M)$ is given up to isotopy by $f_1^b=\pi_I\circ g:M\times I\to I$. The marking on $f^a$ is given by $g^b(\gamma)\beta$ where $\gamma$ is the path from the basepoint to the birth point of $f^b$ which maps by $g^b$ to the death point of $f^a$ and $\beta$ is the path along the Morse point set $p_t$ back to the birth point of $f^a$. This represents the inverse of $f^b$ in $\pi_1(\cF(M),\cE(M))$ by Proposition \ref{cor: inverse of ft}. However, $f_t^{ab}$ is symmetrical since $f^c_t=f_b$ and $f^c_{2-t}=f^a_{1-t}\circ g^b=f^b_t$ and the markings match. So, $\pi_0\cD_0(M)$ is a group and $\widetilde\theta_0$ is a homomorphism.
\end{proof}


\subsection{Computation of second obstruction}

To compute $\pi_0\cD_0(M)$, we need the following. For any group $\pi$ and left $\pi$-module $A$, we consider $A[\pi]:=A\otimes \ZZ[\pi]$, the direct sum of copies of $A$, one for each element of $\pi$. This is a $\pi$-module with the diagonal action of $\pi$ where $\pi$ acts on $\ZZ[\pi]$ by conjugation:
\begin{equation}\label{eq: diagonal action of pi}
	\sigma\left( \sum a_i\otimes g_i\right)=\sum \sigma a_i\otimes \sigma g_i\sigma^{-1}.
\end{equation}
Each conjugacy class in $\pi$ gives a submodule of $A[\pi]$, namely the sum of all $A\otimes \tau$ indexed by the conjugates $\tau$ of some fixed $\sigma\in\pi$. Let $A[e]=A\otimes e$ be the submodule indexed by the identity $e\in\pi$. We consider the quotient module $A[\pi]/A[e]$.

\begin{defn}\label{def: Wh1+ Z}
We define $Wh_1^+(\pi;A)$ to be the additive group of coinvariants of the action of $\pi$ on $A[\pi]/A[e]$:
\[
	Wh_1^+(\pi;A):=(A[\pi]/A[e])_\pi.
\]
When the action of $\pi$ on $A$ is trivial, this is the direct sum of copies of $A$, one for every nonidentity conjugacy class in $\pi$. More generally
\[
	Wh_1^+(\pi;A)\cong \bigoplus A_{C(\sigma_\alpha)}
\]
This is the direct sum over all nonidentity conjugacy classes in $\pi$ of the coinvariants of the action of $C(\sigma_\alpha)$ on $A$ where $C(\sigma_\alpha)$ is the centralizer of a representative $\sigma_\alpha$ of the conjugacy class.
\end{defn}

The main example is $\pi=\pi_1M$ and $A=\ZZ_2\oplus \pi_2M$. For example, if $\pi$ is finite with trivial action on $\pi_2M$ then 
\[
	Wh_1^+(\pi_1M;\ZZ_2\oplus \pi_2M)=\bigoplus_{c-1}\ (\ZZ_2\oplus \pi_2M)
\]
where $c$ is the number of conjugacy classes of elements of $\pi_1M$.

Then the theorem of Hatcher and Wagoner can be restated as follows.

\begin{thm}\label{thm of H-W in terms of D(M)}\cite{HW}
For $\dim M=n\ge6$ we have an isomorphism of groups:
\begin{equation}\label{eq: compute pi0 D0(M)}
	\lambda_0:\pi_0 \cD_0(M)\cong (\ZZ_2\oplus \pi_2M)\otimes \ZZ[\pi_1M]/\ZZ[e]=(\ZZ_2\oplus \pi_2M)[\pi_1M]/(\ZZ_2\oplus \pi_2M)[e].
\end{equation}
Furthermore, this isomorphism is equivariant with respect to the action of $\pi_1M\times\ZZ_2$ given by changing the markings on $\cD_0(M)$ and with the diagonal action of $\pi_1M$ given in \eqref{eq: diagonal action of pi} and the trivial action of $\ZZ_2$. Thus $\pi_0\cD_0(M)_{\pi_1M}=Wh_1^+(\pi_1M;\ZZ_2\oplus \pi_2M)$. Furthermore, we have an exact sequence 
\[
	Wh_1^+(\pi_1M;\ZZ_2\oplus \pi_2M)\xrightarrow{\overline\theta_0} \pi_0\cC(M)\to Wh_2(\pi_1M)\to 0
\]
where $\overline\theta_0$ is the homomorphism induced by $\widetilde\theta_0:\pi_0\cD_0(M)\to \pi_0\cC(M)$.
\end{thm}

The idea is that, by Morse theory, a 1-parameter family of functions gives a family of chain complexes. In this case the chain complexes are acyclic. Hatcher and Wagoner show that this family can be deformed until all critical points are in two indices (the ``Two-index Theorem''). This gives a family of invertible matrices. There is a K-theoretic obstruction to simplifying this family into the constant family of an $m\times m$ identity matrix. The resulting graphic is a disjoint union of $m$ ``eyes'' which can be deformed into one ``eye'' by choosing markings and using the product structure of $\pi_0\cD_0(M)$. 

We now recall the definition of the invariant $\lambda_0$ in \eqref{eq: compute pi0 D0(M)}. Consider a marked lens-shaped model $(f_t)$ for $M^n$ with birth-death points at parameter values $t_0,t_1$. For each $t_0<t<t_1$, $f_t$ has critical points $p_t,q_t$ in indices $k,k+1$ where $k=\lfloor \frac n2\rfloor$ with critical values $f_t(p_t)< \frac12<f_t(q_t)$. The level surface $V_t^n=f_t^{-1}(\frac12)$ has an embedded $k$-sphere $S_q^k$ which is the stable manifold of $q_t$ and an embedded $n-k$ sphere $S^{n-k}_p$ which is the unstable manifold of $p_t$. For $t$ slightly greater than $t_0$ and slightly less than $t_1$, these spheres meet transversely at one point. However, for values of $t$ between $t_0$ and $t_1$, these spheres may meet in more than one point.

\begin{defn}\label{def: framed bordism invariant for D0(M)}\cite{HW}
The \emph{framed bordism invariant} of the marked lens-shaped model $f_t$ is defined as follows. The orientation of the negative eigenspace of $D^2f_t$ at the birth point $b_0$ of $f_{t_0}$ gives an orientation for $S^k_q$ and a normal framing of $S_p^{n-k}$. Thus, their intersection is a framed 1-dimensional submanifold $J^1\subset S^k_q\times (t_0,t_1)$ which has one open component going from $t_0$ to $t_1$. Let $J_0$ be the union of the closed components of $J$. Each point in $J$ represents a trajectory of the negative gradient $-\nabla f_t$ from $q_t$ down to $p_t$. Together with the paths $p_t,q_t$ to the birth point $b_0$ then the marking which is a path from the basepoint of $M\times I^2$ to $b_0$, we obtain a mapping
\[
	J^1\to \Omega M.
\]
Restricting this to $J_0$ gives an element in the framed bordism group
\[
	[J_0]\in \Omega_1^{fr}(\Omega M)=(\ZZ_2\oplus\pi_2M)[\pi_1M].
\]
The open component of $J$ maps to the identity component of $\Omega M$. This makes any framed bordism class in $(\ZZ_2\oplus\pi_2M)[e]$ not well-defined since it could be absorbed into the open component of $J$. Modulo this we obtain the \emph{framed bordism invariant}
\[
	\lambda_0(f_t)\in (\ZZ_2\oplus\pi_2M)[\pi_1M]/(\ZZ_2\oplus\pi_2M)[e].
\]
\end{defn}

If we take the product of two families $f_t,f_t'$ in $\cD_0(M)$, we obtain the disjoint union of the closed components $J_0,J_0'$ of the framed manifolds $J,J'$ for $f_t,f_t'$. Therefore, $\lambda_0$ gives a homomorphism of groups:
\[
	\lambda_0:\pi_0\cD_0(M)\to (\ZZ_2\oplus\pi_2M)[\pi_1M]/(\ZZ_2\oplus\pi_2M)[e].
\]

\begin{lem}\label{lem: framed bordism invariant of D0(M)}\cite{HW}
For $\dim M\ge 6$, $\lambda_0$ is an isomorphism.
\end{lem}

\begin{proof}
In this range of dimensions, $n-k\ge k\ge3$. So, $S_q^k\times (t_0,t_1)$ has dimension $\ge4$ and there is no obstruction to deforming framed $1$-manifolds in $S_q^k\times (t_0,t_1)$. If the framed bordism invariant of $f_t$ is trivial, i.e., if it lies in $(\ZZ_2\oplus\pi_2M)[e]$, each component of $J$ can be absorbed into the open component by connected sum. Since $\dim V\ge 6$, the open component can be straightened making $f_t$ into the identity element of $\pi_0\cD_0(M)$. Thus $\lambda_0$ is a monomorphism.

To show that $\lambda_0$ is surjective, take any generator $(a,b)[\sigma]\in(\ZZ_2\oplus\pi_2M)[\pi_1M]$. This is represented by a framed circle $J_0$ in $S_q^k\times (t_0,t_1)$ together with a map $J_0\to \Omega M$. Since $\dim M\ge6$ this can be realized as a one parameter family of Morse functions. We create two critical points $p_t,q_t$ of index $k,k+1$ in cancelling position. Then create two extra geometric incidences which add and subtract $\sigma$ from the incidence number. Then cancel them again after moving them around using $(a,b)\in \ZZ_2\oplus\pi_2M$. The result has only one geometric incidence between the two critical points and thus they can be cancelled. This creates the desired lens-space module showing that $\lambda_0$ is surjective. 
\end{proof}

We want to extend the framed bordism invariant for $\pi_0\cD_0(M)$ to the case $n=4$ (and $n=5$). In both $n=4,5$, we have $k=2$. So, $S_q^k\times (t_0,t_1)$ has dimension $3$ and the framed 1-dimensional manifold $J_0$ in Definition \ref{def: framed bordism invariant for D0(M)} has integer framing invariant (for the same reason that $\pi_3S^2=\ZZ$). In dimension $n=4$ we cannot prove the analogue of Lemma \ref{lem: framed bordism invariant of D0(M)}. (The case $n=5$ will be discussed further in \cite{Wh1b}.) However, the framed bordism invariant of Definition \ref{def: framed bordism invariant for D0(M)} is still defined and can be refined as follows.

\begin{prop}\label{prop: integer framed bordism invariant}
For $n=4$ or $5$, the framed bordism invariant factors through a well-defined integer framed bordism invariant which is a homomorphism
\[
	\widetilde\lambda_0:\pi_0\cD_0(M)\to (\ZZ\oplus\pi_2M)[\pi_1M]/(\ZZ\oplus\pi_2M)[e].
\]
\end{prop}

\subsection{Dependence on markings} We show that the integer framed bordism invariant $\widetilde\lambda_0$ is $\pi_1M\times\ZZ_2$ equivariant.

\begin{lem}\label{lem: dependence on markings}
For $n=4,5$, the integer framed bordism invariant $\widetilde\lambda_0$ from Proposition \ref{prop: integer framed bordism invariant} is equivariant with respect to the action of $\pi_1M\times\ZZ_2$ given on $\cD_0(M)$ by changing the markings and on $(\ZZ\oplus\pi_2M)[\pi_1M]$ by the diagonal action of $\pi_1M$ as in \eqref{eq: diagonal action of pi} and the sign action of $\ZZ_2$ on $\ZZ$. 
\end{lem}

To clarify the statement, the group of coinvariants of the action of $\pi_1M\times\ZZ_2$ on $(\ZZ\oplus\pi_2M)[\pi_1M]/(\ZZ\oplus\pi_2M)[e]$ is equal to $Wh_1^+(\pi_1M;\ZZ_2\oplus \pi_2M)$. In other words, the sign of the integer framing number of $J_0$ is not intrinsic. It just comes from the marking. This implies the following.

\begin{prop}
For $n=4,5$ the integer framed bordism invariant induces a well-defined mapping
\[
	\overline\lambda:\pi_0\cD(M)\to Wh_1^+(\pi_1M;\ZZ_2\oplus \pi_2M).
\]
\end{prop}

To understand the proof of Lemma \ref{lem: dependence on markings}, it helps to consider a concrete equivalent problem. What happens to the generator of $\pi_3S^2$ when we reverse the orientation for both source and target? When we reverse the orientation of $S^2$, nothing happens! This is because the Hopf map $S^3\to \CC P^1=S^2$ commutes with complex conjugation which is homotopic to the identity on $S^3$ but reverses the orientation of $\CC P^1=S^2$. Reversing the orientation of $S^3$ is negation on $\pi_3S^2$ by definition of its group structure. Therefore, reversing the orientation of both $S^2$ and $S^3$ acts on $\pi_3S^2=\ZZ$ by changing sign.

\begin{proof}[Proof of Lemma \ref{lem: dependence on markings}]
Each closed component of the framed 1-manifold $J_0$ gives an oriented cycle in the loop space of $M\times I^2$ at the birth point of $f_t$. This contributes $(n,a)[\sigma]$ to the integer framing invariant where $n\in \ZZ$, $a\in \pi_2M$ and $\sigma\in \pi_1M$. When we change the orientation of the descending plane of the birth point, we change the orientation of the descending sphere $S^2_q\times (t_0,t_1)$ and the normal orientation of the ascending sphere $S^{n-2}_p\times (t_0,t_1)$. The result is that tangential orientation of the intersection $J\subset S^2_q\times (t_0,t_1)$ is unchanged. However, the self-intersection number of $J$ in $S^2_q\times (t_0,t_1)$ changes sign. This will change the invariant from $(n,a)[\sigma]$ to $(-n,a)[\sigma]$. This is similar to what happens with $\pi_3S^2$.

The other part of the marking is the path from the basepoint of $M\times I^2$ to the birth point of $f_t$. is changed by, say $\gamma\in\pi_1M$, this invariant will change to $(n,\gamma a)[\gamma\sigma\gamma^{-1}]$. The framing number $n$ does not change since it does not use the path to the basepoint.
\end{proof}

\subsection{Involution}\label{ss: involution}

We also have a very useful involution $\varepsilon$ which flips the function $f$ ``upside-down''. It is given by
\[
	\varepsilon f(x,t)=1-f(x,1-t).
\]
This is projection to the second coordinate of the function $\varepsilon g$ given for each $g(x,t)=(h(x,t),f(x,t))$ in $\cC(M)$ by
\[
	\varepsilon g(x,t)=(r^{-1}h(x,1-t),1-f(x,1-t))
\]
where $r$ is the automorphism of $M$ given by $g(x,1)=(r(x),1)$. Since $\varepsilon$ switches the stacking direction for functions (Proposition \ref{prop: stacking functions}) we have 
\[
	\varepsilon(f\bullet f')=\varepsilon(f')\bullet\varepsilon(f).
\]
The involution $\varepsilon$ turns the graphic of a function upside-down. Also, it changes the index $k$ of a Morse critical point $(x,t)$ to index $n+1-k$ (and the critical point moves to $(x,1-t)$). Thus, in the case when $\dim M=n$ is even, this involution sends $\cD(M)$ to $\cD(M)$.

\begin{prop}\label{prop: involution on Wh1+ invariant}
Suppose that $n\ge4$ is even and $(f_t)\in \cD(M)$ with $\overline\lambda(f_t)=(n,\alpha)[\sigma]\in Wh_1^+(\pi_1M;\ZZ_2\oplus \pi_2M)$ where $n\in\ZZ_2$, $\alpha\in\pi_2M$ and $\sigma\in\pi_1M$. Then
\[
	\overline\lambda(\varepsilon(f_t))=(n+w_2(\alpha),-\alpha)[\sigma^{-1}].
\]
where $w_2\in H^2(M;\ZZ_2)$ is the second Stiefel-Whitney class of $M$.
\end{prop}

\begin{proof}
First, choose a marking for $(f_t)$. By linearity of the invariant, we can restrict attention to the contribution of one closed component, say $C$, of $J$, the intersection of the descending sphere $S^k_q\times (t_0,t_1)$ with the ascending sphere $S^k_p\times (t_0,t_1)$ in $V^{2k}\times (t_0,t_1)$.

To compute the invariant for $\varepsilon(f_t)$, we take the same family of functions $f_t$ with the same critical points $p_t,q_t$ but with the negative gradient $-\nabla f_t$ replaced with the positive gradient. The loop representing $\sigma$ in $\overline\lambda(f_t)$ goes from the birth point $b_0$ up to $q_t$, then down through $C$ to $p_t$ and back to $b_0$. The loop for $\overline\lambda(\varepsilon(f_t))$ goes the other way (from $p_t$ up to $q_t$ through $C$). So, $\sigma$ is replaced with $\sigma^{-1}$. 

The tangential orientation of $J$ is unchanged since, by definition, it points from the birth point to the death point on the open component. So, the component $C$ is tangentially oriented in the same direction for both $(f_t)$ and $\varepsilon(f_t)$. But, the $\sigma$ direction is changed. So, the $\pi_2M$ element changes sign from $\alpha$ to $-\alpha$. This element $\alpha\in \pi_2M$ is represented by an embedded $2$-sphere with equator $C$, upper hemisphere $D^2_+$ contained in the descending manifold $D^{k+1}_q\times (t_0,t_1)$ of $q_t$ and lower hemisphere $D^2_-$ contained in the ascending manifold $D^{k+1}_p\times(t_0,t_1)$ of $p_t$. For $\varepsilon(f_t)$ the roles of $D^2_+$, $D^2_-$ are switched.

The framing number $n$ of $C$ in $\overline\lambda(f_t)$ is given by comparing the normal framing of $S_p^k\times (t_0,t_1)$ along $C$, which extends to a normal framing of $D^2_-$ in $D_p^{k+1}\times (t_0,t_1)$, with the tangential framing of $D^2_+$ in $D_q^{k+1}\times (t_0,t_1)$. The framing number $n'$ of $C$ in $\overline\lambda(\varepsilon(f_t))$ is given by comparing the normal framing of $D^2_+$ in $D_q^{k+1}\times (t_0,t_1)$ with the tangential framing of $D^2_-$ in $D_p^{k+1}\times (t_0,t_1)$. So, the sum $n+n'$ compares the normal bundle of $D^2_+$ in $M\times I^2$ with the normal bundle of $D^2_-$. In other words, $n+n'=w_2(D^2_+\cup D^2_-)=w_2(\alpha)$.
\end{proof}

Proposition \ref{prop: involution on Wh1+ invariant} will guide us in the remainder of this paper. We will assume that $M$ is a $4$-manifold with an embedded $2$-sphere with odd $w_2$. For example, we could take connected sum with $\CC P^2$. Using an embedded $2$-sphere we will construct a lens-shaped model for $M$ with invariant $\alpha[\sigma]$ where $\alpha\in \pi_2M$ is the class of the embedded $2$-sphere and $\sigma$ is any element of $\pi_1M$. The explicit construction will allow us to compute the framing number. It will be the degree of the normal bundle of the embedded $2$-sphere in $M$, i.e., its self-intersection number. We do the same using $\sigma^{-1}$ and apply the involution $\varepsilon$. By Proposition \ref{prop: involution on Wh1+ invariant}, the sum will give all nontrivial elements of $Wh_1^+(\pi_1M;\ZZ_2)$.

\section{Constructing pseudoisotopies using two disks}

To prove Theorems \ref{thm A} and \ref{thm B} we need embedded disks: $D_{q}^3,D_{p}^2\subset M^4$ as shown in Figure \ref{Fig: two disks}. We use the fact that the \emph{self-intersection number} of the $2$-sphere $\CC P^1$ in $\CC P^2$, i.e., the degree of its normal bundle, is 1. Conversely, if $M$ contains an embedded $2$-sphere with self-intersection number $1$, it must have the form $M=X\# \CC P^2$.

\begin{lem}\label{lem: existence of two disks} Let $S_0^2\subset M$ be an embedded $2$-sphere with self-intersection number $w$. Given any $\sigma\in \pi_1M$, there are embedded disks $D_{q}^3\subset M^4$ and $D_{p}^2\subset M^4$ with boundaries $S_{q}^2=\partial D_{q}^3$, $S_{p}^1=\partial D_{p}^2$ with the following properties (Figure \ref{Fig: two disks}).
\begin{enumerate}
\item $S_{q}^2\cap D_{p}^2=\{a\}\subset int\,D_{p}^2$
\item $D_{q}^3\cap S_{p}^1=\{b\}\subset int\,D_{q}^3$
\item $D_{q}^3\cap D_{p}^2$ is the disjoint union of
	\begin{enumerate}
	\item a line segment from $a$ to $b$ (denoted $J$)
	\item a circle $S_\ast^1$.
	\end{enumerate}
\item Any path from $b$ to $S_\ast^1$ in $D_{q}^3$ followed by any path from $S_\ast^1$ to $b$ in $D_{p}^2$ is homotopic to (a conjugate of) the chosen element $\sigma\in \pi_1M$.
\item $S_\ast^1$ bounds two $2$-disks $B_{q}^2\subset D_{q}^3$ and $B_{p}^2\subset D_{p}^2$ whose union is homotopic to $S_0^2$.
\item The normal framing of $S_\ast^1$ in $D_{q}^3$ given by any trivialization of the normal bundle of $D_{p}^2$ in $M$ has winding number $w$.
\end{enumerate}
\end{lem}

\begin{figure}[htbp]
\begin{center}
\begin{tikzpicture}

\coordinate (btilde) at (0,0);

\coordinate (C) at (0,0);
\coordinate (Cm) at (-.5,0);
\coordinate (A) at (2,0);
\coordinate (App) at (2,0.1);
\coordinate (Am) at (1.9,0);
\coordinate (Ap) at (2.5,0);
\coordinate (X) at (4.1,1.1);
\coordinate (Y) at (4.1,-1.1);
\coordinate (B) at (3,0);
\coordinate (J) at (2.5,0);
\coordinate (Bp) at (3,0.1);
\coordinate (Ta) at (-1.3,0);
\coordinate (Taa) at (-0,-3);
\coordinate (Tb) at (.3,0);
\coordinate (Tbb) at (0,-3);
\coordinate (Tcc) at (3.2,-3);
\coordinate (Tc) at (3.2,-.4);
\coordinate (Td) at (3.8,-.2);
\coordinate (Tcd) at (3.5,-.3);
\coordinate (Tdd) at (3.3,-4);
\coordinate (Tx) at (-.6,-2.9);
\coordinate (Ty) at (0.2,-2.7);
\coordinate (Tz) at (-.2,-2.8);
\begin{scope}[xshift=-1mm,yshift=2mm]
\draw (1.5,-2) node{$\sigma$};
\draw[<-] (1,-2)..controls (1,-2.5) and (2,-2.5)..(2,-1.8);
\end{scope}
\begin{scope} 
\draw[ thick, color=blue] (-1.3,0) --( -1.3,-2.5);
\draw[ thick, color=blue] (.3,0) --( .3,-2.5);
	\begin{scope} 
	\clip (-1.5,-2.5) rectangle (.5,-3);
	\draw[very thick, color=blue] (-.5,-2.5) ellipse[x radius=8mm,y radius=4mm];
\draw[fill,color=white] (Tz) circle[radius=4mm]; 
\end{scope}
	\draw[very thick,color=blue] (Tx) .. controls (Taa) and (Tdd).. (Td);
	\draw[very thick,color=blue] (Ty) .. controls (Tbb) and (Tcc).. (Tc);
\draw[fill,color=white] (C) ellipse [x radius=3cm,y radius=1.5cm] (B) circle [radius=1cm];
\draw[dashed, color=blue] (-1.3,0) --( -1.3,-2.5);
\draw[dashed, color=blue] (.3,0) --( .3,-2.5);
\end{scope}
\begin{scope} 
	\clip (-1.5,-2.5) rectangle (.5,-2);
	\draw[dashed, color=blue] (-.5,-2.5) ellipse[x radius=8mm,y radius=4mm];
\end{scope}
\begin{scope} 
	\draw[dashed, color=blue] (B) circle [radius=1cm];
	\clip (Ap) rectangle (Y);
	\draw[very thick, color=blue] (B) circle [radius=1cm];
\end{scope}
\draw[fill,color=white] (3.32,-.7) circle[radius=3.5mm];

\draw[very thick] (C) ellipse [x radius=3cm,y radius=1.5cm]; 
\begin{scope} 
	\clip (-2,0) rectangle (1,2.1);
	\draw[fill, color=white] (Cm) ellipse[x radius=8mm,y radius=2cm];
\draw[dashed] (C) ellipse [x radius=3cm,y radius=1.5cm];
	\draw[ thick, color=blue] (Cm) ellipse[x radius=8mm,y radius=2cm];
\end{scope}
\begin{scope} 
	\clip (Am) rectangle (X);
	\draw[thick, color=white, fill] (B) circle [radius=1cm];
	\draw[dashed] (C) ellipse [x radius=3cm,y radius=1.5cm];
	\draw[very thick, color=blue] (B) circle [radius=1cm];
\end{scope}
\begin{scope} 
\draw[dashed, color=red] (Cm) ellipse [x radius=8mm,y radius=4mm];
\clip (-2,0) rectangle (1,-1);
\draw[thick, color=red] (Cm) ellipse [x radius=8mm,y radius=4mm];
\end{scope}
\draw[thick, color=red] (A)--(B) (J)node[above]{$J$}; 
\draw (App) node[left]{$a$}; 
\draw (Bp) node[right]{$b$};
\coordinate (Dp3) at (3,-3); 
\coordinate (Dq2) at (-2.1,-.2); 
\coordinate (Sp2) at (3.5,1); 
\coordinate (Bp2) at (0,2.4); 
\coordinate (S1s) at (0.7,0); 
\coordinate (Sq1) at (-3.5,1.2); 
\coordinate (Bq2) at (-3,-1.5); 
\draw (Bq2) node[left]{$B_{p}^2$};
\draw[->] (Bq2) .. controls (-1,-1.5) and (-.6,-.7)..(Cm);
\draw[color=blue] (Bp2) node{$B_{q}^2$};
\draw (Sq1) node{$\partial D_{p}^2=S_{p}^1$};
\draw[color=blue] (Sp2) node[right]{$\partial D_{q}^3=S_{q}^2$};
\draw[color=blue] (Dp3) node{$D_{q}^3$};
\draw (Dq2) node{$D_{p}^2$};
\draw[color=red] (S1s) node{$S^1_\ast$};
\draw[thick,->] (1.45,-1.3)--(1.2,-.8);
\draw (1.4,-.5) node{$\tilde b$};
\end{tikzpicture}
\caption{Two embedded disks $D_{p}^2, {\color{blue}D_{q}^3}\subset M$ with $D_{p}^2\cap {\color{blue}D_{q}^3}=\color{red}S^1_\ast\coprod J$. In this figure, $D_{p}^2$ is flat and ${\color{blue}D_{q}^3}$ is drawn as a boundary connected sum of two $3$-disks along a path given by $\sigma$. $B_p^2\subset D_p^2$ is a $2$-disk with $\partial B_p^2=\color{red}S^1_\ast$. $\tilde b$ is the normal vector field of $S_p^1$ pointing inward into $D_p^2$.}
\label{Fig: two disks}
\end{center}
\end{figure}
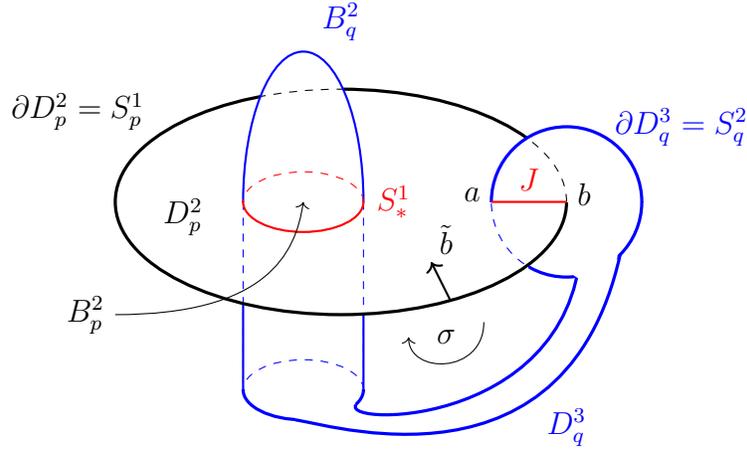

\begin{rem}
We choose a path from $b$ to the base point of $M$ so that the loop at $b$ represents the element $\sigma\in\pi_1M$. We also fix an orientation of the normal bundle of $D_p^2$ in $M$ (giving a normal orientation of $S_p^1$ in $M$) and an orientation of $D_q^3$ so that these orientations agree at $D_{q}^3\cap S_{p}^1=\{b\}$. These choices will give a marking on the lens-shaped model we are constructing (Definition \ref{def: lens-shaped models}.)
\end{rem}

We shall now construct a 1-parameter family of functions $(f_t)$ on $M^4\times I$ assuming we are given embedded disks $D_{q}^3,D_{p}^2\subset M$ as described above and shown in Figure \ref{Fig: two disks}. Start with $f_0=\pi_I$, the projection to $I=[0,1]$. To show that the result is a nontrivial element of $\pi_0\cC(M)$ we need to assume that $w=1$.


\begin{minipage}{0.5\textwidth}
We introduce a pair of cancelling Morse critical points $p_t,q_t$ of indices $2,3$ for $t>t_0$ where $f_{t_0}$ has a birth-death point of index $2$. Choose a path from this birth-death point to the basepoint of $M\times I^2$.
\end{minipage}
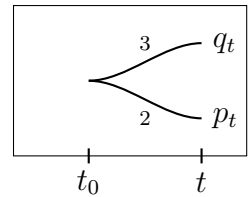
\hfill
\begin{minipage}{0.5\textwidth}
\begin{center}
\begin{tikzpicture}[scale=.5]
\draw (-3,-2)--(-3,2)--(3.2,2)--(3.2,-2)--cycle;
\coordinate (A) at (-1,0);
\coordinate (B) at (2,-1);
\coordinate (B0) at (.5,-1);
\coordinate (C) at (2,1);
\coordinate (C0) at (.5,1);
\coordinate (D1) at (2,-2.2);
\coordinate (D2) at (2,-1.8);
\coordinate (D0) at (2,-2.7);
\coordinate (X1) at (-1,-2.2);
\coordinate (X2) at (-1,-1.8);
\coordinate (X0) at (-1,-2.7);

	\draw[ thick] (A)..controls (0,0) and (1,-1)..(B) node[right]{$p_t$} (B0) node{\tiny 2};
	\draw[ thick] (A)..controls (0,0) and (1,1)..(C) node[right]{$q_t$} (C0) node{\tiny 3};
	\draw[thick] (D2)--(D1) (D0)node{$t$};
	\draw[thick] (X2)--(X1) (X0)node{$t_0$};
\end{tikzpicture}
\end{center}
\end{minipage}%


\begin{minipage}{0.4\textwidth}
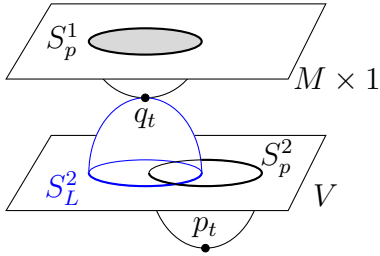
\begin{center}
\begin{tikzpicture}[scale=.5]
\draw (-4.5,-1)--(-3.3,1)--(4,1)--(3,-1)--cycle;
\begin{scope}[yshift=35mm]
\draw (-4.5,-1)--(-3.3,1)--(4,1)--(3,-1)--cycle;
\end{scope}
\coordinate (M) at (4.3,2.5);
\draw (M) node{$M\times 1$};
\coordinate (V) at (4,-.6);
\coordinate (q) at (-.8,2);
\coordinate (p) at (.8,-2);
\coordinate (S2p) at (.8,0);
\coordinate (S2pL) at (2.7,.4);
\coordinate (SL) at (-.8,0);
\coordinate (SL2) at (-3,-.4);
\coordinate (S1) at (-.8,3.5);
\coordinate (S1L) at (-3,3.5);
\draw[blue,thick] (SL) ellipse[x radius=15mm, y radius=3.5mm];
\begin{scope}
\clip (-2.5,2.5) rectangle (1,2);
\draw (S1) ellipse[x radius=15mm, y radius=15mm];
\end{scope}
\draw[fill,gray!30!white] (S1) ellipse[x radius=15mm, y radius=3.5mm];
\draw[thick] (S1) ellipse[x radius=15mm, y radius=3.5mm] (S1L) node{$S_p^1$};
\draw[thick] (S2p) ellipse[x radius=15mm, y radius=3.5mm];
\begin{scope}
\clip (-2.5,0) rectangle (1,2);
\draw[fill, white] (SL) ellipse[x radius=15mm, y radius=20mm];
\draw[blue] (SL) ellipse[x radius=15mm, y radius=20mm];
\end{scope}
\draw[blue] (SL) ellipse[x radius=15mm, y radius=3.5mm];
\draw[fill] (q) circle[radius=1mm] node[below]{$q_t$};
\draw (V) node{$V$};
\draw[blue] (SL2) node{$S_L^2$};
\draw (S2p) ellipse[x radius=15mm, y radius=3.5mm];
\draw (S2pL) node{$S_p^2$};
\begin{scope}
\clip (-2,-1) rectangle (3,-2);
\draw (.8,0) ellipse[x radius=15mm, y radius=20mm];
\end{scope}

\draw[fill] (p) circle[radius=1mm] node[above]{$p_t$};

\end{tikzpicture}
\end{center}
\end{minipage}%
\hfill
\begin{minipage}{0.6\textwidth}
For $t$ slightly greater than $t_0$, the regular levels of the Morse function $f_t$ start at the top: $M^4=M\times1$ as shown in Figure \ref{Fig: two disks}, with an embedded $2$-disk $D_{p}^2\subset M^4$ (shaded on the left), performs surgery on $\partial D_{p}^2=S_{p}^1$ at $q_t$ to produce the intermediate level surface $V^4$ shown in Figure \ref{Fig: V} with an embedded sphere $S_L^2$ which is the stable sphere of $q_t$. $D_p^2\subset M$ descends to $S_p^2\subset V$, the unstable sphere of $p_t$.
\end{minipage}

\begin{figure}[htbp]
\begin{center}
\begin{tikzpicture}
\coordinate (C) at (0,0);
\coordinate (Cm) at (-1,0);
\coordinate (A) at (-4,0);
\coordinate (Ap) at (-4,.2);
\coordinate (B) at (-2.5,0);
\coordinate (Bp) at (-2.5,0.3);
\coordinate (AB) at (-3.4,0);
	\draw[very thick] (-2.5,0) .. controls (-1.5,0) and (-1,-2)..(0,-2);
	\draw[very thick] (1.6,-.5) .. controls (1.6,-1.5) and (1,-2)..(0,-2);
\draw[very thick] (2.8,-.5)..controls (2.8,-1.5) and (2,-3)..(-1,-3);
%
\begin{scope}
	\draw[fill, color=blue!15!white] (C) ellipse[x radius=4cm, y radius =1.5cm];
	\draw[very thick, color=blue] (C) ellipse[x radius=4cm, y radius =1.5cm];
	\draw[thick, color=red] (A)--(B) (AB) node[above]{$\widetilde J$};
	\draw[fill, color=white] (Cm) ellipse[x radius=1.5cm, y radius =.6cm];
	\draw[very thick, color=blue] (Cm) ellipse[x radius=1.5cm, y radius =.6cm];
	\draw (Ap)node[left]{$a$};
	\draw (Bp)node[left]{$\tilde b$};
\end{scope}
\begin{scope} 
\coordinate (D3p) at (0,1);
\coordinate (Sp2) at (-3,1.5);
\coordinate (Sb2) at (-1.2,.3);
\draw[color=blue] (D3p) node{$\widetilde D_{q}^3$};
\draw[color=blue] (Sp2) node{$ S_{q}^2$};
\draw[color=blue] (Sb2) node{$ S_L^2$};
\end{scope}
\coordinate (S) at (2.2,-.5);
\coordinate (Sm) at (1.5,-.5);
\coordinate (Spp) at (3,2);
\begin{scope} 
\draw[color=blue!7!white,fill] (S) ellipse[x radius=6mm,y radius=2mm];
\draw[color=red,thick] (S) ellipse[x radius=6mm,y radius=2mm]; 
\clip (Sm) rectangle (Spp);
\draw[fill,color=blue!7!white] (S) ellipse [x radius=6mm,y radius=8mm];
\draw[very thick] (S) ellipse [x radius=6mm,y radius=8mm];
\draw[color=red,dashed] (S) ellipse[x radius=6mm,y radius=2mm];
\end{scope}
\begin{scope} 
\coordinate (Am) at (-6,0);
\coordinate (Lm) at (-4,-3);
\coordinate (L) at (-1,-3);
\coordinate (Br) at (-1.5,0);
\coordinate (Mm) at (0,-2);
\coordinate (M) at (1,-2);
\coordinate (Mr) at (2,-2);
\coordinate (Mpp) at (-1,-.5);
\coordinate (Mp) at (-1,-1.5);
\draw[very thick] (A)..controls (Am) and (Lm)..(L);
\end{scope}
	\begin{scope} 
	\clip (Cm) ellipse[x radius=1.5cm, y radius =.6cm];
	\draw[very thick] (-2.5,0) .. controls (-1.5,0) and (-1,-2)..(0,-2);
	\end{scope}
	\draw[dashed] (-2.5,0) .. controls (-1.5,0) and (-1,-2)..(0,-2) (1.6,-.5) .. controls (1.6,-1.5) and (1,-2)..(0,-2)  (2.8,-.5)..controls (2.8,-1.5) and (2,-3)..(-1,-3);
	\draw (-2,-2.5) node{$ S_{p}^2$};
	\draw (2.5,.5) node{$ B_{p}^2$};
	\draw[color=red] (3.1,-.45) node{$ S_\ast^1$};
	\draw[color=blue] (3.1,-2) node[right]{$ B_{q}^2$};
	\draw[color=blue,->] (3.1,-2)..controls (2.6,-2) and (2.2,-1)..(2.2,-.5) ;
%
%
%
\end{tikzpicture}
\caption{The level surface $V^4$ obtained by doing surgery on $S_{p}^1\subset M$. The disk $D_{p}^2\subset M$ becomes the sphere $ S_{p}^2\subset V$ which meets $S_L^2\subset V$, the sphere dual to $S_{p}^1$, transversely at $\tilde b$. In $V$, the disk ${\color{blue}D_{q}^3}\subset M$ becomes ${\color{blue}\widetilde D_{q}^3=S^2\times \widetilde J}$ with boundary $\color{blue}\partial\widetilde D_{q}^3=S_L^2\coprod S_{q}^2$. In this figure, $\color{blue}\widetilde D_{q}^3$ is flat.}
\label{Fig: V}
\end{center}
\end{figure}
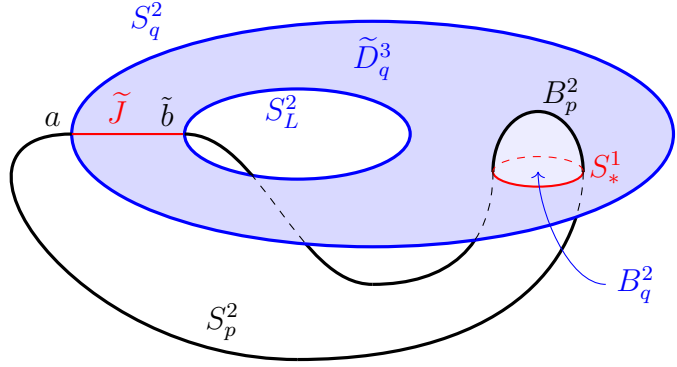

Thus $S_L^2$ is the stable sphere of $q_t$, the set of all points in $V^4$ which are carried by trajectories of $\nabla f_t$ to $q_t$. Points on the surgery circle $S_{p}^1=\partial D_{p}^2$ give normal vector fields for $S_L^2$ in $V^4$ and, in the direction corresponding to the point $b=\partial D_{p}^2\cap D_{q}^3$, a small $2$-sphere in $D_q^3$ centered $b$ in $M$, under the flow generated by $-\nabla f_t$ goes to a sphere $\tilde S_L^2\subset V^4$ parallel to $S_L^2$.

Let $S_{p}^2\subset V^4$ be the unstable sphere of $p_t$. Since $p_t,q_t$ are a cancelling pair of critical points, the two spheres $S_L^2,S_{p}^2$ intersect transversely at one point. We can take $S_{p}^2$ to be the image of $D_{p}^2$ in $V^4$.

\begin{minipage}{0.5\textwidth}
For $t_1\le t\le t_2$ we do an isotopy of the descending sphere of the critical point $q_t$ from $S_L^2$ to $S_{q}^2$ along the cylinder $\widetilde D_{q}^3=S^2\times \widetilde J$. This intersects $S_{p}^2$ along the circle $S_\ast^1$.
\end{minipage}
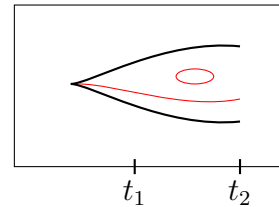
\hfill
\begin{minipage}{0.5\textwidth}
\begin{center}
\begin{tikzpicture}[scale=.5]
\draw (-3,-2.2)--(-3,2.1)--(4.2,2.1)--(4.2,-2.2)--cycle;
\coordinate (A) at (-1.5,0);
\coordinate (B) at (3,-1);
\coordinate (B0) at (.5,-1);
\coordinate (C) at (3,1);
\coordinate (C0) at (.5,1);

	\draw[ thick] (A)..controls (-1,0) and (1,-1.2)..(B);
	\draw[ thick] (A)..controls (-1,0) and (1,1.2)..(C) ;
	\draw[color=red] (-1.3,0)..controls (-.3,0) and (1.5,-.7) .. (3,-.4) (1.8,.2) ellipse[x radius=5mm,y radius=2mm];
	\begin{scope}[yshift=-2mm]
\coordinate (D1) at (3,-2.2);
\coordinate (D2) at (3,-1.8);
\coordinate (D0) at (3,-2.7);
\coordinate (X0) at (0.2,-2.7);
\coordinate (X1) at (0.2,-2.2);
\coordinate (X2) at (0.2,-1.8);
	\draw[thick] (D2)--(D1) (D0)node{$t_2$};
	\draw[thick] (X2)--(X1) (X0)node{$t_1$};
	\end{scope}
\end{tikzpicture}
\end{center}
\end{minipage}%

By construction, any normal framing of $\widetilde D_q^3$ in $V\times I$ gives a trivial framing for  $S_\ast^1$ in $S_{p}^2\times I$ and, by Condition (6), for $n$ odd, any normal framing of $S_p^2\times I$ gives a nontrivial framing of $S_\ast^1$ in $\widetilde D_{q}^3$, the trace of the deformation of $S_L^2$ to $S_{q}^2$. By (5), the trajectories of $-\nabla f_t$ through $S_\ast^1$ gives a $2$-sphere homotopic to $S_0^2$. Also, each of these trajectories, together with the paths $p_t$, $q_t$ to the birth-death point of $f_{t_0}$ followed by the chosen path to the basepoint of $M\times I^2$ gives the chosen element $\sigma\in \pi_1M$.

When $n=1$ and $M=X\# \CC P^2$, $M$ contains a subspace $Y=X-D^4$ with $\pi_1Y=\pi_1M$ and $\CC P^1$ is disjoint from $Y$. So, the projection of the first Postnikov invariant $k_1\in H^3(\pi_1M;\pi_2M)$ to $H^3(\pi_1M;\ZZ)=H^3(\pi_1M;\pi_2\CC P^2)$ is trivial. Thus, in the exact sequence:
\[
	K_3(\ZZ[\pi_1M])\xrightarrow \chi Wh_1^+(\pi_1;\ZZ_2\oplus \pi_2M)\to \pi_1\cP (M)\to Wh_2(\pi_1M)\to 0
\]
the boundary map $\chi$ does not hit the element $\tau[\sigma]\in Wh_1^+(\pi_1M;\pi_2M)$ where $\tau\in\pi_2M$ is the class of $\CC P^1$ since, by \cite{What happens}, this factor of $\chi$ is given by:
\[
	K_3(\ZZ[\pi_1M])\to H_3(\pi_1M;\pi_2M)\xrightarrow{\cap k_1} H_0(\pi_1M;\pi_2M)\to Wh_1^+(\pi_1M;\pi_2M)
\]
which does not hit the summand $Wh_1^+(\pi_1M;\pi_2M)$ generated by $\tau[\pi_1M]$.

\begin{minipage}{0.5\textwidth}
At $t_2$ we can cancel the critical points at a birth-death point to give a marked lens-shaped model with invariant $\tau[\sigma]\in Wh_1^+(\pi_1M;\pi_2M)$ where $\tau\in \pi_2M$ is the homotopy class of $\CC P^1$.
\end{minipage}
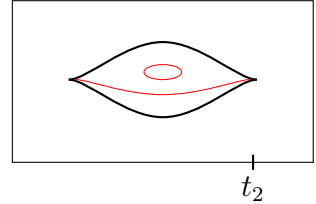
\hfill
\begin{minipage}{0.5\textwidth}
\begin{center}
\begin{tikzpicture}[scale=.5]
\draw (-4,-2.2)--(-4,2.1)--(4,2.1)--(4,-2.2)--cycle;
\coordinate (A) at (-2.5,0);
\coordinate (B) at (2,-1);
\coordinate (B0) at (.5,-1);
\coordinate (C) at (2,1);
\coordinate (C0) at (.5,1);

	\draw[ thick] (-2.5,0)..controls (-2,0) and (-1,-1)..(0,-1) ..controls (1,-1) and (2,0).. (2.5,0);
	\draw[ thick] (-2.5,0)..controls (-2,0) and (-1,1)..(0,1) ..controls (1,1) and (2,0).. (2.5,0);
	\draw[color=red] (-2.3,0)..controls (-2,0) and (-1,-.4) .. (0,-.4).. controls (1,-.4) and (2,0)..(2.3,0) (0,.2) ellipse[x radius=5mm,y radius=2mm];
	\begin{scope}[yshift=-2mm]
\coordinate (D1) at (2.4,-2.2);
\coordinate (D2) at (2.4,-1.8);
\coordinate (D0) at (2.4,-2.7);
	\draw[thick] (D2)--(D1) (D0)node{$t_2$};
	\end{scope}
\end{tikzpicture}
\end{center}
\end{minipage}%

More generally, this same construction proves the following.

\begin{thm}[Theorem \ref{thm A}]\label{prop: realizing Wh1pi2}
Given a compact $4$-manifold $M$ and $\tau\in\pi_2M$ represented by an embedded $2$-sphere and $\sigma\in \pi_1M$, the obstruction element $\tau[\sigma]\in Wh_1^+(\pi_1M; \pi_2M)$ lifts to $\cC (M)$. Furthermore, when the embedded $2$-sphere has self-intersection number $\pm1$, this element of $\pi_0\cC(M)$ is stably nontrivial (survives to a nontrivial element of $\pi_0\cP(M)$).
\end{thm}

\begin{rem}\label{rem: not necessarily nonzero}
The fact that the obstruction element $\tau[\sigma]$ lifts to an element of $\pi_0\cC(M)$ does not imply that its image in $\pi_0\cC(M)$ is nonzero except in the case when the self-intersection number of $\tau$ is $n=\pm1$ since the Postnikov invariant argument given above only works in that case.
\end{rem}

\section{Realizing the framing obstruction}

We use the construction from the previous section to construct elements realizing every element of the framing obstruction group
\[
	Wh_1^+(\pi_1M;\ZZ_2).
\]
However, the comment in Remark \ref{rem: not necessarily nonzero} still applies. Lifting these obstruction elements to $\pi_0\cC(M)$ does not imply that their images are nonzero, except in the following case.

Suppose that $\pi_1M$ is free abelian of finite nonzero rank. Then, as observed by Hatcher, $Wh_1^+(\pi_1M;\ZZ_2)$ will be an infinitely generated group. But $K_3(\ZZ[\pi_1M])$ is finitely generated by the fundamental theorem of K-theory. From the exact sequence
\[
	K_3(\ZZ[\pi_1M])\xrightarrow\chi Wh_1^+(\pi_1M;\ZZ_2)\oplus Wh_1^+(\pi_1M;\pi_2M)\to \pi_0\cP(M)
\]
we can then conclude that there are infinitely many nontrivial elements of $\pi_0\cC(M)$ coming from $Wh_1^+(\pi_1M;\ZZ_2)$. However, these elements might not be $2$-torsion when $M$ is $4$-dimensional.

For nontrivial $\sigma\in\pi_1M$ and $\tau\in\pi_2M$ represented by an embedded $2$-sphere in $M$ with self-intersection number $\pm1$ we realize the obstruction $\tau[\sigma]$ as in Theorem \ref{prop: realizing Wh1pi2} and do it again for $\tau[\sigma^{-1}]$ where we choose markings for our lens-shaped models (lifting them to elements of $\cD_0(M)$) to make these obstructions well-defined (Definition \ref{def: lens-shaped models}). However, the second time, we apply the involution $\varepsilon$ which turns the Morse functions upside down to make the intersection circle $S_\ast^1$ trivially framed in $\widetilde D_{q}^2$ and nontrivially framed in $S_{p}^2\times I$. As in the proof of Proposition \ref{prop: involution on Wh1+ invariant}, the product of these two elements of $\cD_0(M)$ will map to zero in $\pi_2M[\pi_1M]/\pi_2M[e]$ since $\tau[\sigma]-\tau[\sigma]=0$. The integer framed bordism invariant (Proposition \ref{prop: integer framed bordism invariant}) will be $\pm[\sigma]$, a generator of $\ZZ[\pi_1M]/\ZZ[e]$.  This gives:

{
\begin{thm}[Theorem \ref{thm B}]\label{thm: realizing Wh1+Z2}
Let $M=X\#\CC P^2$ be the connected sum of a nonsimply connected $4$-manifold $X$ with $\CC P^2$. Then the integer framed bordism invariant $\widetilde\lambda_0:\pi_0\cD_0(M)\to \ZZ[\pi_1M]/\ZZ[e]$ is surjective. In particular, every element in the image of the second obstruction map $\theta:Wh_1^+(\pi_1M;\ZZ_2)\to \pi_0\cP(M)$ lifts to $\pi_0\cC(M)$.
\end{thm}
}

More generally, an analogous construction gives the following.

\begin{thm}\label{thm: realizing Wh1+Z2 b}
Suppose $M^4$ contains an embedded $2$-sphere with odd self-intersection number. Then every element of the image of $\theta:Wh_1^+(\pi_1M;\ZZ_2)\to\pi_0\cP(M)$ lifts to $\pi_0\cC(M)$.
\end{thm}

\section{Construction of the two disks}

We prove Lemma \ref{lem: existence of two disks} by constructing the two disks $D_{p}^2,D_{q}^3$ first in the special case when the embedded $2$-sphere is $\CC P^1$, then more generally for any embedded $2$-sphere by modifying the construction in the special case. 

The disk $D_{q}^3$ will be the boundary connected sum of two $3$-disks $D_0^3$ and $D_1^3$ where $D_{p}^2,D_0^3, D_1^3$ will all be contained in a punctured $\CC P^2$ contained in $M$.

Let $D_{p}^2$ be the set of all point in $\CC P^2$ with coordinates $[1,z,0]$ where $|z|\le 3$ and let $B_{p}^2$ be the subset of those points $[1,z,0]$ with $|z|\le1$. These are $2$-disks with $\partial B_{p}^2=S_\ast^1$ being the set of all $[1,x,0]$ where $|x|=1$.

Let $S_\ast^1\times D^2$ denote the set of all points in $\CC P^2$ of the form $[1,x,z]$ where $|x|=1$ and $|z|\le 1$. This is the restriction to $S_\ast^1$ of the normal disk bundle of $D_{p}^2$ in $\CC P^2$. Thus $S_\ast^1\times D^2$ meets $D_{p}^2$ transversely along $S_\ast^1$ as shown in Figure \ref{Fig: construction of two disks}.

\begin{figure}[htbp]
\begin{center}
\begin{tikzpicture}
\coordinate (C) at (0,0); %
\coordinate (Cm) at (-.5,0); %
\coordinate (A) at (2,0); %
\coordinate (Am) at (1.9,0); %
\coordinate (Ap) at (2.5,0); %
\coordinate (X) at (4.1,1.1); %
\coordinate (Y) at (4.1,-1.1); %
\coordinate (B) at (3,0); %
\coordinate (J) at (2.5,0); %
	\draw[fill, color=blue!10!white] (B) circle [radius=1cm];

\begin{scope} 
	\clip (Ap) rectangle (Y);
	\draw[very thick, color=blue] (B) circle [radius=1cm];
\end{scope}
\draw[very thick] (C) ellipse [x radius=3cm,y radius=1.6cm]; 

\begin{scope} 
	\clip (Am) rectangle (X);
	\draw[thick, color=white, fill] (B) circle [radius=1cm];
	\draw[fill, color=blue!10!white] (B) circle [radius=1cm];
	\draw[very thick, color=blue] (B) circle [radius=1cm];
\end{scope}
\begin{scope}
	\clip (Am) rectangle (Y);
	\draw[fill,color=white] (C) ellipse [x radius=3cm,y radius=1.6cm];
	\draw[very thick] (C) ellipse [x radius=3cm,y radius=1.6cm];
	\clip (B)circle [radius=1cm];
	\clip (C) ellipse [x radius=3cm,y radius=1.6cm];
	\draw[thick,dashed, color=blue] (B) circle [radius=1cm];
	\draw[very thick] (C) ellipse [x radius=3cm,y radius=1.6cm];
\end{scope}
\begin{scope} 
	\draw[fill, color=blue!10!white] (C) ellipse [x radius=12mm,y radius=6mm];
	\draw[fill, color=white] (0,.3) ellipse [x radius=12mm,y radius=6mm];
\end{scope}
\begin{scope} 
	\clip (0,.3) ellipse [x radius=12mm,y radius=6mm];
	\draw[fill, color=blue!10!white] (0,.3) ellipse [x radius=12mm,y radius=6mm];
	\draw[fill, color=white] (C) ellipse [x radius=12mm,y radius=6mm];
\end{scope}
\begin{scope} 
\draw[dashed, color=red] (C) ellipse [x radius=12mm,y radius=6mm];
\clip (-2,0) rectangle (1.5,-1);
\draw[thick, color=red] (C) ellipse [x radius=12mm,y radius=6mm];
\end{scope}
\draw[very thick, color=blue] (0,.3) ellipse [x radius=12mm,y radius=6mm] (1.2,.3)--(1.2,0)(-1.2,.3)--(-1.2,0);
\draw[dashed, color=blue] (0,-.3) ellipse [x radius=12mm,y radius=6mm];
\draw[thin, color=blue] (1.2,-.3)--(1.2,0)(-1.2,-.3)--(-1.2,0);
\draw[thick, color=red] (A)--(B) (J)node[below]{$J$}; 
\coordinate (Sp2) at (3.5,1); %
\coordinate (Bq2) at (-3,-1.5); %
\draw (-2.7,-1.1) node[left]{$D_{p}^2$};
\draw[color=blue] (Sp2) node[right]{$ D_0^3$};
\draw[color=red] (-1.7,0) node{$S^1_\ast$};
\draw[color=blue] (0,1.2) node{$S_\ast^1\times D^2$};
\end{tikzpicture}
\caption{The $2$-disk $D_{p}^2\subset \CC P^1\subset \CC P^2$ with subdisk $B_{p}^2$. The boundary $\color{red}\partial B_{p}^2=S_\ast^1$ is in red. In blue are $\color{blue}S_\ast^1\times D^2$, meeting $D_{p}^2$ transversely along $\color{red}S_\ast^1$ and $\color{blue}D_0^3$ meeting $\partial D_{p}^2$ tranversely at one point with $\color{red}J=D_0^3\cap D_{p}^2$ in red.}
\label{Fig: construction of two disks}
\end{center}
\end{figure}
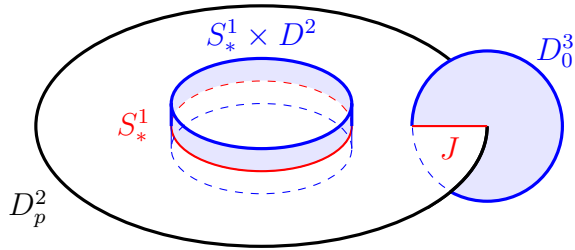

Let $D_0^3$ be the $3$-disk in $\CC P^2$ which meets $\partial D_{p}^2=S^1_{p}$ transversely at one point and so that $D_0^3\cap D_{p}^2$ is the line segment $J$ as shown in Figure \ref{Fig: construction of two disks}.

To construct $D_1^3$, we take the twisted circle $\widetilde S_\ast^1$ in $\partial(S_\ast^1\times D^2)$ given by 
\[
	\widetilde S_\ast^1=\{[1,x,x]\in \CC P^2\,|\, |x|=1\}.
\]
This bounds a $2$-disk in $\CC P^2$ in the complement of $\CC P^1$ by taking the set of all $[t,x,x]\in \CC P^2$ with $|x|=1$ and $t\in [0,1]$. This is a $2$-disk since $[t,x,x]=[t\overline x,1,1]$. The union of $B_{p}^2$ with this $2$-disk is a $2$-sphere homotopic to $\CC P^1$. Use this $2$-disk to attach a handle to $S_\ast^1\times D^2$ to obtain a $3$-disk $D_1^3$ which meets $D_{p}^2$ transversely along $S_\ast^1$.

Finally, let $D_{q}^3$ be given by taking the boundary connect sum of the two $3$-disks $D_1^3, D_0^3$ along a path in $M$ given by the chosen element $\sigma\in \pi_1M$. The result is as shown in Figure \ref{Fig: two disks} proving Lemma \ref{lem: existence of two disks} in the special case then $S_0^2=\CC P^1$.

\begin{proof}[Proof of Lemma \ref{lem: existence of two disks} in general] Let $S_0^2\subset M$ be an embedded $2$-sphere with self-intersection number $n$. We follow the same outline as above using Figure \ref{Fig: construction of two disks}. 

Let $D_{p}^2$ in Figure \ref{Fig: construction of two disks} be the southern hemisphere of $S_0^2$ with subdisk $B_{p}^2\subset D_{p}^2\subset S_0^2$ and $S_\ast^1=\partial B_{p}^2$. For the circle bundle $\partial S_\ast^1\times D^2\to S_\ast^1$ take the section $\widetilde S_\ast^1$ with winding number $n$. This is the section which extends to a section $B_{q}^2$ of the normal circle bundle over $S_0^2\backslash B_{p}^2$. We can use this section to attach a $3$-dimensional $2$-handle to $S_\ast^1\times D^2$. This gives a $3$-disk $D_1^3$ which intersects $D_{p}^2$ transversely along $S_\ast^1$ so that $B_{q}^2\cup B_{p}^2\subset D_1^3\cup D_{p}^2$ is homotopic to $S_0^2$. Let $D_{q}^3$ be the boundary connected sum of $D_1^3$ with a $3$-disk $D_0^3$ given as before using the path $\sigma$ to do the connected sum. This produces the disks $D_{p}^2$ and $D_{q}^3$ with all the required properties.
\end{proof}

\section*{Acknowledgements}

The author would like to thank Sander Kupers for encouraging him to write up these results and for commenting on earlier versions of this manuscript. He also thanks Bj\o rn Jahren and Tom Goodwillie for useful correspondences. The author is indebted to Allen Hatcher for explaining pseudoisotopy theory to him and for posing the question of what happens in low dimensions. The author is supported by the Simons Foundation.

\end{document}